\documentclass{article}
\usepackage{amsmath}
\usepackage{amsfonts}
\usepackage{amsthm}
\usepackage{amssymb}
\usepackage{color}
\usepackage{graphicx}
\usepackage{float}
\usepackage{subfigure}
\usepackage[font=small]{caption}
\usepackage{tikz}

\newtheorem{theorem}{Theorem}[section]
\newtheorem{lemma}[theorem]{Lemma}

\newtheorem{observation}[theorem]{Observation}

\theoremstyle{definition}
\newtheorem{definition}[theorem]{Definition}
\theoremstyle{remark}

\newcommand\remove[1]{}

\def\f2{\mathbb{F}_2}

\begin{document}

\title{Distortion of embeddings of binary trees into diamond graphs}

\author{Siu~Lam~Leung, Sarah~Nelson, Sofiya~Ostrovska, and Mikhail~Ostrovskii}

\date{\today}
\maketitle

\begin{abstract}

Diamond graphs and binary trees are important examples in the
theory of metric embeddings and also in the theory of metric
characterizations of Banach spaces. Some results for these
families of graphs are parallel to each other, for example
superreflexivity of Banach spaces can be characterized both in
terms of binary trees (Bourgain, 1986) and diamond graphs
(Johnson-Schechtman, 2009). In this connection, it is natural to
ask whether one of these families admits uniformly bilipschitz
embeddings into the other. This question was answered in the
negative by Ostrovskii (2014), who left it open to determine the
order of growth of the distortions. The main purpose of this paper
is to get a sharp up-to-a-logarithmic-factor estimate for the
distortions of embeddings of binary trees into diamond graphs,
and, more generally, into diamond graphs of any finite branching
$k\ge 2$. Estimates for distortions of embeddings of diamonds into
infinitely branching diamonds are also obtained.

\end{abstract}

\smallskip

\medskip

\begin{large}

\section{Introduction}

Binary trees and diamond graphs play an important role in the
theory of metric embeddings and metric characterizations of
properties of Banach spaces, see
\cite{Bou86,BC05,GNRS04,JS09,Klo14,LN04,NR03,Ost11,Ost14,Reg13}
and also presentations in the books \cite{Ost13,Pis16}.

Some results for these families of graphs are parallel to each
other, for example superreflexivity of Banach spaces can be
characterized both in terms of binary trees (Bourgain
\cite{Bou86}) and diamond graphs (Johnson-Schechtman \cite{JS09}).
In this connection, it is natural to ask whether these families of
graphs admit bilipschitz embeddings with uniformly bounded
distortions one into another. In one direction the answer is
clear: The fact that diamond graphs do not admit uniformly
bilipschitz embeddings into binary trees follows immediately from
the combination of the result of Rabinovich and Raz
\cite[Corollary 5.3]{RR98} stating that the distortion of any
embedding of an $n$-cycle into any tree is $\ge\frac{n}3-1$, and
the observation that large diamond graphs contain large cycles
isometrically. As for the opposite direction, it was proved in
\cite{Ost14} that binary trees do no admit uniformly bilipschitz
embeddings into diamond graphs. The goal of this paper is to get a
sharp-up-to-a-logarithmic-factor estimate for the distortions of
embeddings of binary trees into diamond graphs, and, more
generally, into diamond graphs of any finite branching $k\ge 2$.
In addition, estimates for distortions of embeddings of diamonds
into infinitely branching diamonds {are obtained}.

\section{Definitions and the main result}

To begin with, let us present the necessary definitions.

\begin{definition}\label{D:trees} A {\it binary tree of  depth $n$}, denoted $T_n$, is a
finite graph in which each vertex is represented by a finite
(possibly empty) sequence of $0$'s and $1$'s, of length at most
$n$. Two vertices in $T_n$ are adjacent if the sequence
representing  one of them is obtained from the sequence
representing the other by adding one term on the right. (For
example, vertices corresponding to $(1,1,1,0)$ and $(1,1,1,0,1)$
are adjacent.) Vertices which correspond to sequences of length
$k$ are called vertices of $k$-th {\it generation}. The vertex
corresponding to the empty sequence is called  a {\it root}. If a
sequence $\tau$ is an initial segment of the sequence $\sigma$ we
say that $\sigma$ is a {\it descendant} of $\tau$ and that $\tau$
is an {\it ancestor} of $\sigma$. See Figure \ref{F:Tree} for a
sketch of $T_3$.
\end{definition}

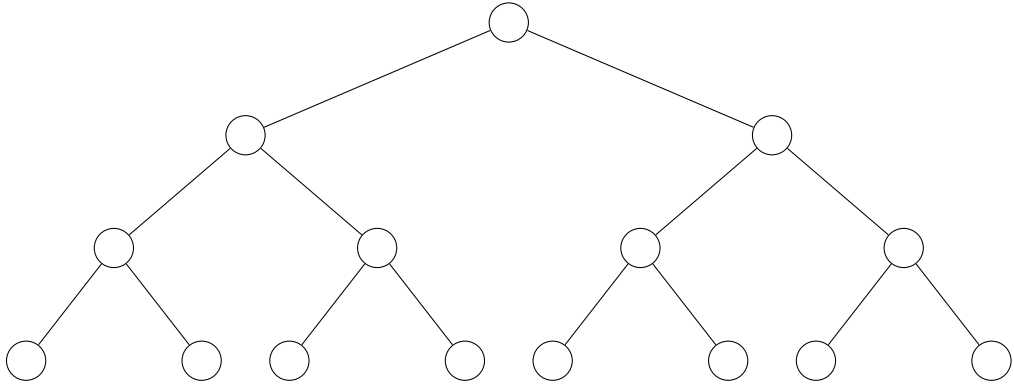
\begin{figure}
{
\begin{tikzpicture}[level/.style={sibling distance=70mm/#1}]
\node [circle,draw] (z) {\hbox{~~}}
  child {node [circle,draw] (a) {\hbox{~~}}
    child {node [circle,draw] (b) {\hbox{~~}}
      child {node [circle,draw]  {\hbox{~~}}}
      child {node [circle,draw]  {\hbox{~~}}}
    }
    child {node [circle,draw] (g) {\hbox{~~}}
      child {node [circle,draw]  {\hbox{~~}}}
      child {node [circle,draw]  {\hbox{~~}}}
    }
  }
  child {node [circle,draw] (a) {\hbox{~~}}
    child {node [circle,draw] (b) {\hbox{~~}}
      child {node [circle,draw]  {\hbox{~~}}}
      child {node [circle,draw]  {\hbox{~~}}}
    }
    child {node [circle,draw] (g) {\hbox{~~}}
      child {node [circle,draw]  {\hbox{~~}}}
      child {node [circle,draw]  {\hbox{~~}}}
      }
    };
\end{tikzpicture}
} \caption{The binary tree of depth $3$, that is, $T_3$.}
\label{F:Tree}
\end{figure}

\begin{definition}[\cite{GNRS04}]\label{D:Diamonds}
Diamond graphs $\{D_n\}_{n=0}^\infty$ are defined inductively as
follows: The {\it diamond graph} of level $0$ is denoted by $D_0$.
It has two vertices joined by an edge. The {\it diamond graph}
$D_n$ is obtained from $D_{n-1}$ as follows. Given an edge $uv\in
E(D_{n-1})$, it is replaced by a quadrilateral $u, a, v, b$, with
edges $ua$, $av$, $vb$, $bu$. See Figure \ref{F:Diamond2} for a
sketch of $D_2$.
\end{definition}

All graphs considered in this paper are endowed with the shortest
path distance: the distance between any two vertices is the number
of edges in a shortest path between them.

\begin{definition}\label{D:dist}
Let $M$ be a finite metric space and $\{R_n\}_{n=1}^\infty$ be a
sequence of finite metric spaces with increasing cardinalities.
The {\it distortion} $c_R(M)$ of embeddings of $M$ into
$\{R_n\}_{n=1}^\infty$ is defined as the infimum of $C\ge 1$ for
which there is $n\in\mathbb{N}$, a map $f:M\to R_n$, and a number
$r=r(f)>0$ - called {\it scaling factor} - satisfying the
condition:
\begin{equation}\label{E:MapDist}\forall u,v\in M\quad rd_M(u,v)\le
d_{R_n}(f(u),f(v))\le rCd_M(u,v).\end{equation}
\end{definition}

Therefore, $c_D(T_n)$ is the infimum of distortions of embeddings
of the binary tree $T_n$ into diamond graphs. Our main result is
expressed by the following assertion:

\begin{theorem}\label{T:main} There exists a constant $c>0$ such that
\[c\,\frac{n}{\log_2n}\le c_D(T_n)\le 2n\]
for all $n\ge2$.
\end{theorem}

\begin{figure}
\begin{center}
\begin{tikzpicture}
  [scale=.25,auto=left,every node/.style={circle,draw}]
  \node (n1) at (16,0) {\hbox{~~~}};
  \node (n2) at (5,5)  {$1$};
  \node (n3) at (11,11)  {$1$};
  \node (n4) at (0,16) {$2$};
  \node (n5) at (5,27)  {$1$};
  \node (n6) at (11,21)  {$1$};
  \node (n7) at (16,32) {\hbox{~~}};
  \node (n8) at (21,21)  {$1$};
  \node (n9) at (27,27)  {$1$};
  \node (n10) at (32,16) {$2$};
  \node (n11) at (21,11)  {$1$};
  \node (n12) at (27,5)  {$1$};

  \foreach \from/\to in {n1/n2,n1/n3,n2/n4,n3/n4,n4/n5,n4/n6,n6/n7,n5/n7,n7/n8,n7/n9,n8/n10,n9/n10,n10/n11,n10/n12,n11/n1,n12/n1}
    \draw (\from) -- (\to);

\end{tikzpicture}
\caption{Diamond $D_2$ in which generations of vertices are
shown.}\label{F:Diamond2}
\end{center}
\end{figure}
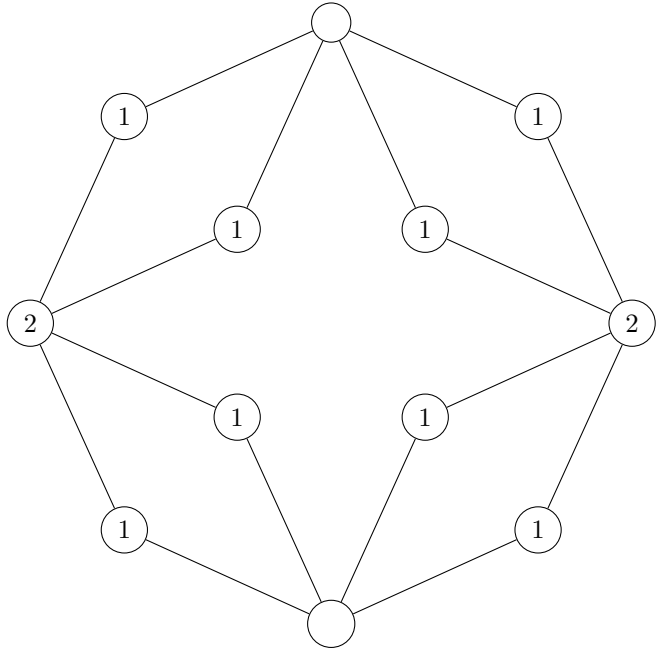

In recent years \cite{BCDKRSZ16+,LR10,OR16,OR16+} we see an
increasing interest in diamonds of high branching (see Definition
\ref{D:BranchDiam}). In view of this, we prove versions of Theorem
\ref{T:main} for such graphs (Theorems \ref{T:FinBranch} and
\ref{T:InfinBranch}).

\begin{definition}\label{D:BranchDiam}
Fix $k\in\mathbb{N}\cup\{\infty\}$, $k\ge 2$. Let $D_{0,k}$ be a
graph consisting of two vertices joined by one edge. The graph
$D_{n+1,k}$ is obtained from $D_{n,k}$ if we replace each edge
$uv$ in $D_{n,k}$ by a set of $k$ paths of length $2$ joining $u$
and $v$. We call the graphs $D_{n,k}$ {\it diamonds of branching}
$k$ if $k$ is finite and {\it diamonds of infinite branching} if
$k=\infty$.

Call one of the vertices of $D_{0,k}$ the {\it top} and the other
the {\it bottom}. Define the {\it top} and the {\it bottom} of
$D_{n,k}$ as vertices which evolved from the top and the bottom of
$D_{0,k}$, respectively. A {\it subdiamond} of $D_{n,k}$ is a
subgraph which evolved from an edge of some $D_{m,k}$ for $0\le
m\le n$. The {\it top} and {\it bottom} of a subdiamond of
$D_{n,k}$  are defined as the vertices of the subdiamond which are
the closest to the top and bottom of  $D_{n,k}$, respectively.
\end{definition}

It is can be noticed  that $D_n=D_{n,2}$. Let $c_{(D,k)}(M)$
denote the distortion of embeddings of a finite metric space $M$
into $\{D_{n,k}\}$, as in  Definition \ref{D:dist}. The next
generalization of Theorem \ref{T:main} holds:

\begin{theorem}\label{T:FinBranch} If $k$ is finite, then there exists $c(k)>0$ such that
\[c(k)\frac{n}{\log_2n}\le c_{(D,k)}(T_n)\le 2n\]
for all $n\ge 2$.
\end{theorem}

For infinitely branching diamonds,  the following weaker version
of Theorem \ref{T:main} is valid:

\begin{theorem}\label{T:InfinBranch} There exists
constant $c(\infty)>0$ such that \[c(\infty)\sqrt{n}\le
c_{(D,\infty)}(T_n)\le 2n.\]
\end{theorem}

We refer to \cite{BM08} for graph-theoretical terminology and to
\cite{Ost13} for terminology of the theory of metric embeddings.

\section{Estimates from above}\label{S:Above}

Since $D_n$ is isometric to a subset of $D_{n,k}$ whenever $k\ge
2$, it suffices to prove the estimate from above for the binary
diamonds $\{D_n\}$.

\begin{proof}[Proof of $c_D(T_n)\le 2n$] Observe that the diamond $D_k$ contains
isometrically the tree which is customarily denoted $K_{1,2^k}$.
This tree has $2^k+1$ vertices, and one of the vertices is
incident to the remaining $2^k$ vertices. In fact, one can easily
establish by induction that the bottom of the diamond $D_k$ has
degree $2^k$, and the bottom together with all of its neighbors
forms the desired tree.
\medskip

Choose $k$ in such a way that $2^k+1\ge 2^{n+1}-1$, where
$2^{n+1}-1$  is the number of vertices in $T_n$.
\medskip

Now,  map the root of $T_n$ to the bottom of $D_k$ and all of
other vertices of $T_n$ to distinct vertices adjacent to the
bottom. Denote the obtained map by $F_n$ and the vertex set of
$T_n$ by $V(T_n)$. We claim that the following inequalities are
true
\begin{equation}\label{E:distor}
\forall u,v\in V(T_n)\quad \frac1{n}d_{T_n}(u,v)\le
d_{D_k}(F_n(u),F_n(v))\le 2d_{T_n}(u,v),
\end{equation}
yielding $c_D(T_n)\le 2n$.

Indeed, the right-hand side inequality follows from the fact that
any distance between two distinct vertices in $K_{1,2^k}$ does not
exceed $2$.

To justify the left-hand side inequality consider the two cases:

(1) One of the vertices, say $u$, is the root of $T_n$. Then
$d_{D_k}(F_n(u),F_n(v))=1$ and $d_{T_n}(u,v)\le n$. The left-hand
side inequality in this case follows.

(2) Neither $u$ nor $v$ is the root of $T_n$. Then
$d_{D_k}(F_n(u),F_n(v))=2$ and $d_{T_n}(u,v)\le 2n$. The left-hand
side inequality follows in this case, too.\end{proof}

\section{Estimates from below}\label{S:below}

\subsection{Diamonds of finite branching}

Observe that Theorem \ref{T:main} is a special case of Theorem
\ref{T:FinBranch}. For this reason, only   the lower estimate of
Theorem \ref{T:FinBranch} has to be proved.

\begin{proof}[Proof of $c_{(D,k)}(T_n)\ge
c(k)\,\frac{n}{\log_2n}$] Fix an integer $k$ ($2\le k<\infty$) for
the whole proof and omit from most of our notation dependence on
$k$, as it is clear that almost all of the introduced objects
depend on $k$. If $\alpha_{n}=3c_{(D,k)}(T_n)$, then there exists
a map $F_n$ of $V(T_n)$ into $V(D_{m(n),k})$ for some
$m(n)\in\mathbb{N}$ satisfying \eqref{E:MapDist} with
$C=\alpha_{n}$ and the scaling factor being an integer power of
$2$, that is,
\begin{equation}\label{E:MapDist2}\forall u,v\in V(T_n)\quad 2^{p(n)}d_{T_n}(u,v)\le
d_{D_{m(n),k}}(F_{n}(u),F_{n}(v))\le \alpha_n
2^{p(n)}d_{T_n}(u,v)\end{equation} for some $p(n)\in\mathbb{Z}$.
If $p(n)<0$, we compose the map $F_n$ with the natural map of
$D_{m(n),k}$ into $D_{m(n)-p(n),k}$. As the latter map increases
all distances into $2^{-p(n)}$ times, the resulting map has
scaling factor equal to $1$. Therefore, one may assume without
loss of generality that $p(n)\ge 0$.\medskip

Now our goal is to show that the existence of  $n,r,
d\in\mathbb{N},$ such  that the condition $1\le r<n$ is satisfied
simultaneously with the following three inequalities:

\begin{equation}\label{E:CannotLeave}2^{d-1}> \alpha_n 2^{p(n)}(r+1),\end{equation}

\begin{equation}\label{E:NotEnoughVert} (2k)^{d-p(n)}< 2^r,\end{equation}

\begin{equation}\label{E:Reach} 2^d<2^{p(n)} (n-r),\end{equation}
leads to a contradiction.
\medskip

We introduce generations of vertices in diamonds, including the
case $k=\infty$, as follows. Generations are labelled recursively
from the end in the following way. {\it Generation number} $1$ in
$D_{m,k}$ is the set of vertices which appeared in the last step
of the construction of $D_{m,k}$. Further, {\it generation number}
$2$ is the set of vertices which appeared in the previous step of
the construction, and so on. In this way, one obtains $m$ {\it
generations}, while the two original vertices do not belong to any
of the generations. See Figure \ref{F:Diamond2} for generations in
$D_2$. This definition leads to the following:

\begin{observation}\label{O:Generat} Let $m\in\mathbb{N}$ and
$k\in\mathbb{N}\cup\{\infty\}$, $k\ge 2$.\medskip

\noindent{\bf (1)} Let $v$ be a vertex of generation number $d$ in
$D_{m,k}$, where $d\in\{1,\dots,m\}$. Then the
$2^{d-1}$-neighborhood of $v$ consists of two subdiamonds of
diameter $2^{d-1}$ each, pasted together at $v$.\medskip

\noindent{\bf (2)} Let $Z_d$ be the set of all vertices of
generation number $d$ in $D_{m,k}$. Then the connected components
of $D_{m,k}\backslash Z_d$ have diameters strictly less than
$2^{d}$.
\end{observation}

Recall that generations for vertices of $T_n$ are defined in the
standard way: the {\it generation} of a vertex in $T_n$ is its
distance to the root.\medskip

Let $n$ and $r$ be satisfying the conditions above, so $1\le r<n$.
Consider any vertex $\tau_{n-r}$ of generation $n-r$ in $T_n$. For
a path $\tau_0,\dots,\tau_{n-r}$  joining the root $\tau_0$ and
$\tau_{n-r}$, inequality \eqref{E:MapDist2} implies that
$d_{D_{m(n),k}}(F_n(\tau_i),F_n(\tau_{i+1}))\le \alpha_n 2^{p(n)}$
and $d_{D_{m(n),k}}(F_n(\tau_0),F_n(\tau_{n-r}))\ge
2^{p(n)}(n-r)$. Combining these inequalities with condition
\eqref{E:Reach} and Observation \ref{O:Generat}~{\bf (2)}, one
concludes that there exists $i\in\{0,\dots,n-r\}$ such that
$d_{D_{m(n),k}}(F_n(\tau_i),v)\le \alpha_n2^{p(n)}$ for some $v$
of generation $d$ in $D_{m(n),k}$.
\medskip

By inequalities \eqref{E:MapDist2}, \eqref{E:CannotLeave} and
Observation \ref{O:Generat}~{\bf(1)}, $F_n$ maps descendants of
$\tau_i$ (in $T_n$) of generations $i+1,\dots,i+r$ (note that
$i+r\le n$) into the union of two subdiamonds of diameter
$2^{d-1}$ each, pasted together at $v$. To obtain a contradiction
with \eqref{E:NotEnoughVert} we need the following lemma:

\begin{lemma}\label{L:Entropy_k} The cardinality of a $2^{p(n)}$-separated set - i.e., a set satisfying $d(u,v)\ge 2^{p(n)}$ for any $u\ne v$ -
in a subdiamond of $D_{m,k}$ of diameter $2^q$ does not exceed
$k\cdot (2k)^{q-p(n)}$ if $q\ge p(n)$.
\end{lemma}

\begin{proof} It is easy to see that  each subdiamond of $D_{m,k}$ of diameter
$2^{p(n)}$ contains at most $k$ vertices out of each
$2^{p(n)}$-separated set. The number of subdiamonds of diameter
$2^{p(n)}$ in a diamond of diameter $2^q$ is equal to the number
of edges in the diamond of diameter $2^{q-p(n)}$. This number of
edges is $(2k)^{q-p(n)}$, because in each step of the construction
of diamonds the diameter doubles and the number of edges is
multiplied by $(2k)$.
\end{proof}

This  contradicts \eqref{E:NotEnoughVert} because, on one hand,
the vertex $\tau_i$ has more than $2^r$ descendants in the next
$r$ generations, and the images of these descendants, by the
bilipschitz condition \eqref{E:MapDist2}, should form a
$2^{p(n)}$-separated set. On the other hand, Lemma
\ref{L:Entropy_k} implies that a $2^{p(n)}$-separated set in a
union of two diamonds of diameters $2^{d-1}$ does not exceed
$2\cdot k\cdot (2k)^{d-1-p(n)}=(2k)^{d-p(n)}$.
\medskip

Since $\{c_{(D,k)}(T_n)\}_{n=2}^\infty$ and
$\{\frac{n}{\log_2n}\}_{n=2}^\infty$ are sequences of positive
numbers and $\alpha_n=3c_{(D,k)}(T_n)$, to prove the existence of
a constant $c(k)>0$ such that $c_{(D,k)}(T_n)\ge
c(k)\,\frac{n}{\log_2n}$, it suffices to show that the existence
of the subsequence of values of $n$ for which $\alpha_n=o(\frac
n{\log_2n})$ leads to a contradiction. This will be done by
demonstrating  that the existence of such a subsequence implies
the existence of $n$, $r$ and $d$ satisfying $1\le r<n$ and
\eqref{E:CannotLeave}--\eqref{E:Reach}. Let us rewrite
inequalities \eqref{E:CannotLeave}--\eqref{E:Reach} as:

\begin{equation}\label{E:rCannotLeave_k}2^{d-p(n)}> 2\alpha_n (r+1),\end{equation}

\begin{equation}\label{E:rNotEnoughVert_k}(2^{(d-p(n))})^{\log_2(2k)}< 2^r,\end{equation}

\begin{equation}\label{E:rReach_k} 2^{d-p(n)}<n-r.\end{equation}

Set $r=r(n)=\lceil \log_2(2k)\cdot\log_2n\rceil$. Since $k$ is
fixed, for sufficiently large $n$, one has $n-r> 2$. Define
$d=d(n)\in \mathbb{N}$ to be the largest integer for which
\eqref{E:rReach_k} holds. It has to be pointed out that with this
choice of $d$, the inequality $2^{d-p(n)}>\frac{n}4$ holds when
$n$ is sufficiently large. Since for our choice of $r$ we have
$2\alpha_n (r(n)+1)=o(n)$ for the corresponding subsequence of
values of $n$, it is clear that, for sufficiently large $n$ in the
subsequence, the condition \eqref{E:rCannotLeave_k} is also
satisfied. It remains to observe that, with the described choice
of $r$, the inequality \eqref{E:rNotEnoughVert_k} follows from
\[(2^{d-p(n)})^{\log_2(2k)}<n^{\log_2(2k)}.\]
Since by virtue of \eqref{E:rReach_k},  $2^{d-p(n)}<n$, the last
inequality is obvious.
\end{proof}

\subsection{Diamonds of infinite branching}

In this case, the  methods based on the upper bounds for
cardinalities of $2^{p(n)}$-separated sets in subdiamonds are not
applicable  since the cardinalities are infinite. Consequently,
the method of \cite{Ost14}, which gives weaker estimates, but
works in the case of infinite branching, will be employed.

\begin{proof}[Proof of Theorem \ref{T:InfinBranch}]  Since the upper estimate has already been established in Section \ref{S:Above}, to complete the proof it has to be shown
 that $c_{(D,\infty)}(T_n)\ge c(\infty)\sqrt{n}$ for some
constant $c(\infty)>0$.

If $\alpha_{n}=3c_{(D,\infty)}(T_n)$, then  there exists a map
$F_n$ of $V(T_n)$ into $V(D_{m(n),\infty})$ for some
$m(n)\in\mathbb{N}$ satisfying \eqref{E:MapDist} with
$C=\alpha_{n}$ and the scaling factor being an integer power of
$2$, that is,
\begin{equation}\label{E:MapDist3}\forall u,v\in V(T_n)\quad 2^{p(n)}d_{T_n}(u,v)\le
d_{D_{m(n),\infty}}(F_{n}(u),F_{n}(v))\le \alpha_n
2^{p(n)}d_{T_n}(u,v)\end{equation} for some $p(n)\in\mathbb{Z}$.
With the help of the same argument as in  Theorem
\ref{T:FinBranch},  it may be  assumed that $p(n)\ge 0$.
\medskip

Since $\{c_{(D,\infty)}(T_n)\}_{n=1}^\infty$ as well as
$\{\sqrt{n}\}_{n=1}^\infty$ are sequences of positive numbers, it
suffices to prove that the inequality of the form
$c_{(D,\infty)}(T_n)\ge c\sqrt{n}$ holds for some $c>0$ and
sufficiently large $n$.
\medskip

Assume that $n>9$ and denote by $d=d(n)$  the largest integer
satisfying

\begin{equation}\label{E:Reach_d}
2^d<2^{p(n)}\cdot\left\lfloor\frac{n}3\right\rfloor
\end{equation}

 Consider any vertex
$\tau_{\lfloor\frac{n}3\rfloor}$ of generation
$\left\lfloor\frac{n}3\right\rfloor$ in $T_n$. Let
$\tau_0,\dots,\tau_{\lfloor\frac{n}3\rfloor}$ be a path joining
the root $\tau_0$ and $\tau_{\lfloor\frac{n}3\rfloor}$ in $T_n$.
Inequality \eqref{E:MapDist3} implies:
\[d_{D_{m(n),\infty}}(F_n(\tau_i),F_n(\tau_{i+1}))\le \alpha_n
2^{p(n)}\] and
\[d_{D_{m(n),\infty}}(F_n(\tau_0),F_n(\tau_{\lfloor\frac{n}3\rfloor}))\ge
2^{p(n)}\left\lfloor\frac{n}3\right\rfloor.\]  By combining these
inequalities with condition \eqref{E:Reach_d} and Observation
\ref{O:Generat}~{\bf (2)}, we conclude that there exists
$i\in\{0,\dots,\left\lfloor\frac{n}3\right\rfloor\}$ such that
\begin{equation}\label{E:CloseGend} d_{D_{m(n),\infty}}(F_n(\tau_i),v)\le \alpha_n2^{p(n)}\end{equation} for
some $v$ of generation $d$ in $D_{m(n),\infty}$.

Inequality \eqref{E:MapDist3}  together with \eqref{E:CloseGend}
implies that descendants of $\tau_i$ of generation $n$ in $T_n$
will be mapped onto vertices whose  distances from $v$ are at
least $(n-i)2^{p(n)}-\alpha_n2^{p(n)}$. One has:

\begin{equation}\label{E:Inter} (n-i)2^{p(n)}-\alpha_n2^{p(n)}\ge\left(\frac23\,n-\alpha_n\right)2^{p(n)}\end{equation}
If $\alpha_n\ge \frac{n}3$, the conclusion of the theorem holds
with $c(\infty)=\frac19$. Therefore, assume that $\alpha_n\le
\frac{n}3$. In this case, the right-hand side of \eqref{E:Inter}
is not less than
\[\frac{n}3\,2^{p(n)}\stackrel{\eqref{E:Reach_d}}{>} 2^d>2^{d-1}.\]

Thence, on each path joining $\tau_i$ with one of its descendants
of generation $n$ (in $T_n$) there is a vertex which is mapped by
$F_n$ outside the union of two subdiamonds of height $2^{d-1}$
with the common vertex at $v$.
\medskip

Let  $x_1$ and $x_2$ different from $v$ the tops/bottoms of the
subdiamonds mentioned in the previous paragraph. The statement
about the paths mentioned in the previous paragraph implies that
on each path joining $\tau_i$ with one of its descendants (in
$T_n$) of generation $n$ there is a vertex, the $F_n$-image of
which is at distance at most $\alpha_n2^{p(n)}$ from either $x_1$
or $x_2$. In fact, it is clear that this condition holds for the
first vertex on the path whose $F_n$-image is outside the union of
the subdiamonds.

Now, let us fix such a path and estimate from below the generation
$r$ of the first vertex on this path, whose $F_n$-image is outside
the union of the subdiamonds. It can be seen by using inequality
\eqref{E:MapDist3} that the earliest generation $r$ for which it
is possible for such image to be outside the union of the
subdiamonds has to satisfy
\[(r-i)2^{p(n)}\alpha_n+2^{p(n)}\alpha_n>2^{d-1},\]
whence
\[(r-i)\ge\frac{2^{d-1}}{2^{p(n)}\alpha_n}-1.\]

The choice of $d$ - see the line preceding \eqref{E:Reach_d} -
implies that
\begin{equation}\label{E:d}2^{d+1}\ge
2^{p(n)}\cdot\left\lfloor\frac{n}3\right\rfloor,\end{equation} and
hence
\[(r-i)\ge \frac{1}{4\alpha_n}\left\lfloor\frac{n}3\right\rfloor-1.\]

Next, consider four different descending paths in $T_n$ starting
at different descendants of $\tau_i$ of generation $(i+2)$. Along
each of these paths we pick the first vertex whose $F_n$-image is
outside the union of the two subdiamonds. Let $v_1,v_2,v_3$, and
$v_4$ be the picked vertices and suppose that these vertices
belong to generations $r_1,r_2,r_3$, and $r_4$, respectively, in
$T_n$.
\medskip

First,  assume that $r_j>i+2$ for $j=1,2,3,4$, while the case
where $r_j=i+2$ for some $j$ will be considered at the very end of
the proof. By the argument above, each $r_j$ satisfies:
\[(r_j-i)\ge \frac{1}{4\alpha_n}\left\lfloor\frac{n}3\right\rfloor-1.\]
Therefore the pairwise distances between vertices $v_1,v_2,v_3$,
and $v_4$ are at least:

\[r_{j_1}+r_{j_2}-2i-2\ge
\frac{1}{2\alpha_n}\left\lfloor\frac{n}3\right\rfloor-4,\quad
j_1,j_2\in\{1,2,3,4\},~j_1\ne j_2. \]

The argument above implies that under the assumption $r_j>i+2$ the
image $F_n(v_j)$ is at distance at most $2^{p(n)}\alpha_n$ to
either $x_1$ or $x_2$. As a result, at least two of these images
are at distance at most $2\cdot 2^{p(n)}\alpha_n$ from each other.
Using \eqref{E:MapDist3} one concludes:

\[2^{p(n)}\left(\frac{1}{2\alpha_n}\left\lfloor\frac{n}3\right\rfloor-4\right)\le
2\cdot2^{p(n)}\alpha_n\]

It is easy to see that this inequality implies that $\alpha_n\ge
c\sqrt{n}$ for some constant $c>0$ and sufficiently large $n$.
\medskip

Now comes  the case where $r_j=i+2$ for some $j\in\{1,2,3,4\}$. In
this case, the distance between $F_n(v_j)$ and $v$, on one hand,
is $> 2^{d-1}$, and, on the other hand, by \eqref{E:CloseGend} and
\eqref{E:MapDist3}, it is $\le 3\alpha_n2^{p(n)}$. This leads to
$3\alpha_n2^{p(n)}>2^{d-1}$. Combining with \eqref{E:d}, one
obtains:
\[3\alpha_n2^{p(n)}>2^{p(n)-2}\cdot\left\lfloor\frac{n}3\right\rfloor.\]
Thus,
$\alpha_n\ge\frac1{12}\cdot\left\lfloor\frac{n}3\right\rfloor$,
yielding  that in this case
$c_{(D,\infty)}(T_n)\ge\frac1{36}\cdot\left\lfloor\frac{n}3\right\rfloor$.
This inequality is sufficient for our purposes. This completes the
proof of Theorem \ref{T:InfinBranch}.
\end{proof}

\section{Acknowledgement}

The work on this paper was inspired by a question on the order of
distortion of embeddings of binary trees into diamonds, asked by
Assaf Naor during the National AMS meeting in Baltimore, Maryland,
2014. The question about analogues of the main result for diamonds
with high branching was asked by Thomas Schlumprecht during the
Workshop in Analysis and Probability, Texas A\&M University, 2016.

The last-named author gratefully acknowledges the support by
National Science Foundation DMS-1201269 and by Summer Support of
Research program of St. John's University during different stages
of work on this paper. Part of the work on this paper was done
when the last-named author was a participant of the NSF supported
Workshop in Analysis and Probability, Texas A\&M University, 2016.

We would like to thank the referee for careful reading of the
paper and for suggesting important improvements of our
presentation.

\end{large}

\renewcommand{\refname}{

\textsc{Department of Mathematical Sciences, Kent State
University, Kent, Ohio 44242}
\par \textit{E-mail address}:
\texttt{sleung1@kent.edu}\par\smallskip

\textsc{Department of Mathematics and Statistics, Hunter College,
CUNY, New York, NY 10065}\par \textit{E-mail address}:
\texttt{sarah.nelson07@yahoo.com}\par\smallskip

\textsc{Department of Mathematics, Atilim University, 06836 Incek,
Ankara, TURKEY} \par \textit{E-mail address}:
\texttt{sofia.ostrovska@atilim.edu.tr}\par\smallskip

\textsc{Department of Mathematics and Computer Science, St. John's
University, 8000 Utopia Parkway, Queens, NY 11439, USA} \par
  \textit{E-mail address}: \texttt{ostrovsm@stjohns.edu} \par

\end{document}